\documentclass[a4paper]{amsart}
\usepackage [english]{babel}
 \usepackage{amsmath}\usepackage{amssymb}
 \usepackage{graphicx}
 \usepackage{subfigure}
 \usepackage{color}

\newtheorem{theorem}{Theorem}[section]
\newtheorem{corollary}[theorem]{Corollary}

\newtheorem{lemma}[theorem]{Lemma}

\theoremstyle{definition}
\newtheorem{definition}[theorem]{Definition}
\newtheorem{remark}[theorem]{Remark}
\newtheorem{example}[theorem]{Example}

\newcommand{\R}{\mathbb{R}}
\newcommand{\N}{\mathbb{N}}

\newcommand{\D}{\partial}
\newcommand{\cl}[1]{\overline{#1}}
\newcommand{\dif}[1]{\,\mathrm{d}#1}
\newcommand{\ldif}[1]{\frac{\dif{\phantom{#1}}}{\dif{#1}}\,}

\newcommand{\wt}[1]{\widetilde{#1}}

\newcommand{\X}{\times}

\newcommand{\ee}{\mathrm{e}}

\newcommand{\f}{\mathfrak{f}}
\newcommand{\g}{\mathfrak{g}}

\def\overstrike#1#2{{\setbox0\hbox{$#2$}\hbox to \wd0{\hss
      $#1$\hss}\kern-\wd0\box0}}

\DeclareMathOperator{\sign}{\mathrm{sign}}

\title[Periodic perturbations]%: Periodic perturbations with delay of coupled ODEs]%
{Periodic perturbations of a class of scalar second order functional differential equations}
\author[A.\ Calamai]{Alessandro Calamai}
\address{Alessandro Calamai, 
Dipartimento di Ingegneria Civile, Edile e Architettura,
Universit\`{a} Politecnica delle Marche
Via Brecce Bianche
I-60131 Ancona, Italy}%
\email{calamai@dipmat.univpm.it}%
\author[M.P.\ Pera]{Maria Patrizia Pera}
\address{Maria Patrizia Pera,
Dipartimento di Matematica e Informatica ``Ulisse Dini'',
Universit\`a degli Studi di Firenze,
Via S.\ Marta 3, I-50139 Florence, Italy}%
\email{mpatrizia.pera@unifi.it}
\author[M.\ Spadini]{Marco Spadini}
\address{Marco Spadini,
Dipartimento di Matematica e Informatica ``Ulisse Dini'',
Universit\`a degli Studi di Firenze,
Via S.\ Marta 3, I-50139 Florence, Italy}%
\email{marco.spadini@unifi.it}
\keywords{Functional differential equations, Branches of periodic solutions, Linear chain trick, Brouwer degree.}
\subjclass[2010]{34K13,  34C25}

\thanks{The authors are members of the Gruppo Nazionale per l'Analisi Mate\-ma\-tica, la Probabilit\`a e le loro Applicazioni (GNAMPA) of the Istituto Nazionale di Alta Mate\-ma\-tica (INdAM)}

\begin{document}
\begin{abstract}
We study, by means of a topological approach, the forced oscillations of second order functional retarded differential equations subject to periodic perturbations. We consider a delay-type functional dependence involving a gamma probability distribution. By a linear chain trick we obtain a first order system of ODE's whose $T$-periodic solutions correspond to those of the functional equation.
\end{abstract}
% \nocite{*}
\maketitle

\section{Introduction and setting of the problem}
In this paper we investigate the forced oscillations of a class of second order functional retarded differential equations subject to periodic perturbations.
Namely, we study perturbations of \emph{scalar, second order functional differential equations} of the following form:
\begin{equation}\label{eq:nonpert}
 \ddot x(t) = g\left(x(t),\dot x(t), \int_{-\infty}^t\gamma_a^b(t-s) \varphi(x(s),\dot x(s))\dif{s}\right),
\end{equation}
where $g\colon \R^3\to\R$ is a continuous map and $\varphi\colon \R^2\to\R$ is locally Lipschitz.
The integral kernel $\gamma_a^b$, for $a>0$ and $b\in \N\setminus\{0\}$ is the gamma probability distribution
\begin{equation}\label{def:gamma}
 \gamma_a^b(s)=\frac{a^bs^{b-1}\ee^{-as}}{(b-1)!}\;\;\text{for $s\geq 0$},\qquad \gamma_a^b(s)=0\;\;\text{for $s< 0$,}
\end{equation}
with mean $b/a$ and variance $b/a^2$.

 The particular dependence on the past of the solution that is considered here is one that naturally arises in many contexts (not necessarily for second order equations) see, e.g., \cite{Bur2005,BuTr82,Far73,Far74,Rua2006,Smi2011}.  In the equation we examine, the delay is spread along the whole history but more concentrated at a given point in the past. We can interpret this, from a probabilistic point of view, as a delay that follows a gamma-type distribution with a given mean (and variance). This approach seems to be reasonable in contexts when the delay cannot be measured with precision, but its mean value and variance are known, or it is genuinely spread in time as, e.g., in \cite{CuRuWe00,CuRuWe03}.
Notice, in particular, that letting $a$ and $b$ tend to infinity in such a way that the quotient $r:=a/b$ remains constant (for instance, put $a=nr$ and $b=n$ and let $n\to\infty$), one concentrates the memory effect close to the delay $r$. Indeed, at least in the sense of distributions, \eqref{eq:nonpert} approximates the  second order delay differential equation
$$\ddot x(t) = g\left(x(t),\dot x(t),  \varphi\big(x(t-r),\dot x(t-r)\big)\right).$$

Observe that the function $\gamma_a^b$ defined in \eqref{def:gamma} is continuous for $b=2,3,\ldots$ but not for $b=1$. However, also in the latter case, assuming $x$ in $C^1$, the function $t\mapsto\int_{-\infty}^t\gamma_a^1(t-s) \varphi(x(s),\dot x(s))\dif{s}$ is continuous. Then the right hand side of \eqref{eq:nonpert} is continuous as well; consequently, any $C^1$ solution of \eqref{eq:nonpert} is actually of class $C^2$.
It is also worth noticing (and we will use this fact later) that, since in the integral one has $t-s\geq0$, then
\[
 \int_{-\infty}^t\gamma_a^1(t-s) \varphi(x(s),\dot x(s))\dif{s}=
 \int_{-\infty}^ta\ee^{-a(t-s)} \varphi(x(s),\dot x(s))\dif{s},
\]
so that the function $t\mapsto\int_{-\infty}^t\gamma_a^1(t-s) \varphi(x(s),\dot x(s))\dif{s}$ is actually in $C^1$.
  
\smallskip
Given $T>0$, we consider the following $T$-periodic perturbation of~\eqref{eq:nonpert}:
\begin{equation}\label{eq:ODEpert}
 \ddot x(t) = g\left(x(t),\dot x(t), \int_{-\infty}^t\gamma_a^b(t-s) \varphi(x(s),\dot x(s))\dif{s}\right)+\lambda f\big(t,x(t),\dot x(t)\big),
\end{equation}
in which we assume that the map $f\colon \R^3\to\R$ is continuous and $T$-periodic in the first variable, and $\lambda$ is a nonnegative real parameter. Our main purpose is to investigate the set of $T$-periodic solutions of \eqref{eq:ODEpert}.
Here, given $\lambda \geq 0$, by a \emph{$T$-periodic solution} (on $\R$) of \eqref{eq:ODEpert} we mean a $C^2$ function $x\colon\R\to \R$ of period $T$ that satisfies identically equality \eqref{eq:ODEpert}.

Roughly speaking, we explicitly construct a scalar function $\Phi$ whose change of sign implies the existence of a connected set, called a \emph{``branch''},  of nontrivial pairs $(\lambda, x)$ -- with $x$ a $T$-periodic solution of \eqref{eq:ODEpert} corresponding to $\lambda\in [0,\infty)$ -- that emanates out of the set of zeros of $\Phi$ and whose closure is not compact.

We also give sufficient conditions yielding the multiplicity of  $T$-periodic solutions of \eqref{eq:ODEpert} for $\lambda> 0$ small. Such conditions are essentially based upon the notion of \emph{ejecting set} (see, e.g.,~\cite{FPS00}).

The methods that we employ are topological in nature and are based on the Brouwer degree. A familiarity with this notion, however, is not necessary to use our main results since its application is restricted only to the proofs. Nevertheless, we provide a brief summary. 

Our results are so-to-speak dual to those of \cite{CaPeSp18}, 
where periodic perturbations containing delay terms are applied to 
scalar, second order ODE's.
The study, by means of topological methods, of the branching and multiplicity of periodic solutions of periodically perturbed equations, is now a well-investigated subject  in the case of ODE's both in Euclidean spaces and on manifolds
(see, e.g., \cite{FuPeSp11,FuSp09}).
The case of ODE's perturbed with delayed forcing terms is also studied in the literature, although not so broadly (see, e.g.,~\cite{CaPeSp18,CaPeSp19}).
However, the presence of delay terms in the \emph{unperturbed} equation \eqref{eq:nonpert} is peculiar of the problem addressed here. In spite of the apparent similarities, a different approach is called upon in order to manage the unperturbed equation.

In the undelayed case, that is for periodically forced second-order scalar autonomous ODE's, both
the problems of existence and multiplicity of periodic solutions are quite classical.
However, this is still the subject of active research by the mathematical community.
There are many approaches that have been successfully pursued to get multiplicity results: among the others let us mention here, e.g., the recent contributions \cite{BFG,BZ13,CaPeSp17,FZ},
the survey papers \cite{Maw09,Maw14} as well as the monograph~\cite{Fonda} and the references therein.

The strategy adopted in this paper is inspired to \cite{Spa22} and can be summarized as follows. First we construct a system of $b+2$, $T$-periodically perturbed first-order ordinary differential equations whose $T$-periodic solutions correspond to that of \eqref{eq:ODEpert}. We then use known results about these perturbed systems to ensure the existence of a branch as sought when the topological degree of the unperturbed field $G$ has nonzero degree and finally, we show, by homotopy techniques, that the change of sign of $\Phi$ imply that $G$ has nonzero degree. 

It should be noted that the result just described does not guarantee the existence of forced oscillations even for very small values of $\lambda$. For this reason, we also provide a nonlocal condition, based on an inequality of J.\ A.\ Yorke \cite{Yo69}, implying that the branch projects nontrivially onto $[0,\infty)$.

In order to develop a better understanding of the nature of the branch of $T$-periodic solutions and its relation with multiplicity results we discuss, following \cite{BiSp15} and \cite{Spa22}, a method to visualize, in finite dimension, a homemorphic set that retains all the relevant properties. We illustrate the procedure with a numeric example.

\section{The linear chain trick}\label{sect:lct}

As pointed out in the introduction, a crucial remark is that the $T$-periodic solutions of \eqref{eq:ODEpert} and those of the following system of $b+2$ ordinary differential equations correspond in some sense:
\begin{equation}\label{eq:expanded}
 \dot \xi = G(\xi)+\lambda F(t,\xi),
\end{equation}
where $\lambda \geq 0$, $\xi=(u, v_0,v_1,\ldots,v_b)\in \R^{b+2}$, and the maps
$G\colon \R^{b+2}\to\R^{b+2}$ and
$F\colon \R\times\R^{b+2}\to\R^{b+2}$ are respectively defined as
\begin{equation}\label{eq:vectorFieldG}
\begin{split}
G(u, v_0,v_1,\ldots,v_b)=\:&\\
\Big(v_0,g(u,v_0,v_b)&,a\big(\varphi(u,v_0)-v_1\big),a(v_1-v_2)\ldots,a(v_{b-1}-v_b)\Big)
\end{split}
\end{equation}
and
\[
 F(t,u, v_0,v_1,\ldots,v_b) = \Big(0,f(t,u,v_0),\underbrace{0,\ldots,0}_{\text{$b$ times}}\Big),
\]
with  $g$, $\varphi$ and $f$  as in equation \eqref{eq:ODEpert}.
Clearly, $G$ and $F$ are continuous maps. By a $T$-periodic solution (on $\R$) of system \eqref{eq:expanded} we mean a $C^1$ function $\xi\colon\R\to \R^{b+2}$ of period $T$ that satisfies \eqref{eq:expanded} identically. 

\medskip
In our main results we will be concerned with the Brouwer degree of the map~$G$
(see Section \ref{sect:degree} for details).
Our first task will be to derive a formula for the computation of the degree of $G$ in terms of the real-valued function $\Phi\colon \R\to\R$, given by 
\begin{equation}\label{eq:defphi}
\Phi(u)=g(u,0,\varphi(u,0)).
\end{equation}

\begin{remark}\label{rem:zero}
 Observe that if $(\bar u, v_0,v_1,\ldots,v_b)\in G^{-1}(0)$ then $\Phi(\bar u)=0$, $v_0=0$ and $v_1=v_2=\ldots=v_b=\varphi(\bar u,0)$. Conversely, for any $\bar u\in\Phi^{-1}(0)$, then $G(\bar u,0,\varphi(\bar u,0)\ldots,\varphi(\bar u,0))=0$.
\end{remark}

Given a bounded open interval $(\alpha,\beta)\subseteq \R$, define the open subset $W^*=(\alpha,\beta)\times\R^{b+1}$ of $\R^{b+2}$. We have:

\begin{theorem}\label{thm:degree}
 Suppose that $\Phi(\alpha)\cdot\Phi(\beta)<0$.
 Then $G$ is admissible for the degree in $W^*$ and we have $\deg(G,W^*)\neq 0$.
\end{theorem}

The proof  of Theorem \ref{thm:degree} is postponed to Section \ref{sect:degree}.

\medskip

Let us show now how a ``linear chain trick'' 
(see, e.g., \cite{Smi2011,Spa22}) can be used to prove the correspondence between $T$-periodic solutions of the second order equation~\eqref{eq:ODEpert} and of the first order system \eqref{eq:expanded}.
 
Let us introduce some notation. By $C_T^n(\R^s)$, $n=0,1,2$, we will denote the Banach space of the $T$-periodic $C^n$ maps $x\colon\R\to\R^s$  with the  the standard norm
\[
\|x\|_{C^n} = \sum_{i=0}^n\max_{t\in\R}|x^{(i)}(t)|.
\]
Here, $x^{(i)}$ denotes the $i$-th derivative of $x$, in particular $x^{(0)}$ coincides with $x$.

%Observe that any $\R^s$-valued, $C^n$, $T$-periodic function is necessarily \hbox{$C^n$-bounded,} hence it belongs to $C_T^n(\R^s)$.

\begin{theorem}\label{thm:equivSol2}
Suppose $x_0$ is a $T$-periodic solution of \eqref{eq:ODEpert}, and let
\begin{equation*}
\left\{\begin{array}{l}
 y_0(t):= \dot x_0(t),\\
  y_i(t):=\int_{-\infty}^t \gamma_a^i(t-s)\varphi(x_0(s),y_0(s))\dif{s},\quad i=1,\ldots,b
 \end{array}\right.
\end{equation*}
for $t \in \R$. Then, $(x_0,y_0,y_1\ldots,y_b)$ is a $T$-periodic solution of \eqref{eq:expanded}. 
\end{theorem}

\begin{proof}
First notice that, since $x_0\in C_T^2(\R)$ and $\varphi\colon\R^2\to\R$ is continuous (we are actually assuming it to be locally Lipschitz), it is not difficult to prove that the functions $y_i$, $i=0,\ldots,b$, are in $C_T^1(\R)$.

Moreover, we claim that, for any $t \in \R$,
\begin{equation*}
\left\{\begin{array}{l}
 \dot x_0(t)= y_0(t),\\
 \dot y_0(t)= g\big(x_0(t),y_0(t),y_b(t)\big)+\lambda f\big(t,x_0(t),y_0(t)\big),\quad \lambda\geq 0,\\
 \dot y_1(t)= a\big(\varphi(x_0(t),y_0(t))-y_1(t)), \\
 \dot y_i(t)=a\big(y_{i-1}(t)-y_i(t)\big),\quad  i=2,\ldots,b.
 \end{array}\right.
\end{equation*}
Indeed, the first equality follows by definition.
Since $x_0$ is a solution of \eqref{eq:ODEpert}, for $ \lambda\geq 0$ we have
\begin{align*}
 \dot y_0(t) & = \ddot x_0(t)= g\left(x_0(t),\dot x_0(t), \int_{-\infty}^t\gamma_a^b(t-s) \varphi(x_0(s),\dot x_0(s))\dif{s}\right)+\lambda f\big(t,x_0(t),\dot x_0(t)\big) \\
& =g\big(x_0(t),y_0(t),y_b(t)\big)+\lambda f\big(t,x_0(t)y_0(t)\big), \quad \forall t \in \R. \end{align*}
Now observe that, when $i=1$, we have for any $t \in \R$, \[
 y_1(t):=\int_{-\infty}^t \gamma_a^1(t-s)\varphi(x_0(s),y_0(s))\dif{s}
 =\int_{-\infty}^t a\ee^{a(t-s)}\varphi(x_0(s),y_0(s))\dif{s}.
\]
Thus, taking the derivative under the integral sign,
\[
   \ldif{t}y_1(t) = \varphi(x_0(t),y_0(t))-a\int_{-\infty}^t \ee^{a(t-s)}\,\varphi(x_0(s),y_0(s))\dif{s}
\]
so that
\[
 \dot y_1(t)= a\big(\varphi(x_0(t),y_0(t))-y_1(t)\big).
\]
Also, for $i=2,\ldots,b$,  we get for any $t \in \R$,
  \begin{align*}
   \ldif{t}y_i(t) &= \gamma_a^i(0)\,\varphi(x_0(t),y_0(t))+\int_{-\infty}^t \ldif{t}\gamma_a^i(t-s)\,\varphi(x_0(s),y_0(s))\dif{s}\\
   &=\int_{-\infty}^t a\left(\gamma_a^{i-1}(t-s)-\gamma_a^i(t-s)\right)\,\varphi(x_0(s),y_0(s))\dif{s}\\
   &=a\left(\int_{-\infty}^t \gamma_a^{i-1}(t-s)\,\varphi(x_0(s),y_0(s))\dif{s}-\int_{-\infty}^t \gamma_a^i(t-s)\,\varphi(x_0(s),y_0(s))\dif{s}\right)\\
   &=a\left(y_{i-1}(t)-y_i(t)\right)
     \end{align*}
proving our claim.

Finally, to see that $(x_0,y_0,y_1\ldots,y_b)$ is $T$-periodic, observe that so are $x_0$ and $y_0$, and with the change of variable $\sigma=s-T$ we get
\begin{align*}
 y_i(t+T)=&\int_{-\infty}^{t+T} \gamma_a^i(T+t-s)\varphi(x_0(s),y_0(s))\dif{s}=\\&\int_{-\infty}^{t} \gamma_a^i(t-\sigma)\varphi(x_0(\sigma+T),y_0(\sigma+T))\dif{\sigma}
 =\\&\int_{-\infty}^{t} \gamma_a^i(t-\sigma)\varphi(x_0(\sigma),y_0(\sigma))\dif{\sigma}=y_i(t),
\end{align*}
for $i=1,\ldots,b$ and any $t \in \R$, and this completes the proof.
\end{proof}

Conversely, we have the following:
\begin{theorem}\label{thm:equivSol1}
 Suppose that $(x_0,y_0,y_1\ldots,y_b)$ is a $T$-periodic solution of \eqref{eq:expanded}, then $x_0$ is a $T$-periodic solution of \eqref{eq:ODEpert}. 
\end{theorem}

The proof of Theorem \ref{thm:equivSol1} is based on the following technical lemma on linear systems of ODEs, see \cite[Lemma 3.3]{Spa22},
cf.\ also \cite[Prop.\ 7.3]{Smi2011}.
The proof is omitted.

\begin{lemma}[\cite{Spa22}, Lemma 3.3]\label{lem:unicsol}
Given any continuous and bounded function 
$z_0\colon\R\to\R$ 
and any nonzero number $a$, 
there exists a unique $C^1$ solution $z=(z_1,\dots,z_b), \, z_i \colon \R\to \R, \, i=1,\ldots,b,$
 of the system in $\R^{b}$
 \begin{equation*}\label{eq:Lsystem}
  \dot z_i(t)=a\big(z_{i-1}(t)-z_i(t)\big),
 \end{equation*}
which is bounded in the $C^1$ norm. This solution is given by
 \begin{equation*}\label{eq:defy}
  z_i(t)=\int_{-\infty}^t \gamma_a^i(t-s)z_0(s)\dif{s},\quad i=1,\ldots,b.
 \end{equation*}
\end{lemma}

\begin{remark}
Observe in particular that, for $i=1$ in the previous lemma, one has 
\[
 z_1(t)=\int_{-\infty}^t \gamma_a^1(t-s)z_0(s)\dif{s}=
 a\int_{-\infty}^t \ee^{a(t-s)}z_0(s)\dif{s},
\]
so that $z_1$ is actually a $C^1$ function, and thus so are all the $z_i$'s for $i=2,\ldots,b$.
\end{remark}

\begin{proof}[Proof of Theorem \ref{thm:equivSol1}]
Let $(x_0,y_0,y_1\ldots,y_b)$ be a $T$-periodic solution of~\eqref{eq:expanded}, and define $z_0(t)=\varphi(x_0(t),y_0(t))$ for $t\in\R$. Observe that $z_0$ is bounded and (at least)  continuous.
Thus, by Lemma \ref{lem:unicsol}, 
   \begin{equation*}
   y_i(t)=\int_{-\infty}^t \gamma_a^i(t-s)\varphi(x_0(s),y_0(s))\dif{s}\quad i=1,\ldots,b,
 \end{equation*}
 is the unique solution of class $C^1$ of
\begin{equation*}
\left\{\begin{array}{l}
 \dot y_1(t)= a\big(\varphi(x_0(t),y_0(t))-y_1(t)), \\
 \dot y_i(t)=a\big(y_{i-1}(t)-y_i(t)\big),\quad  i=2,\ldots,b.
 \end{array}\right.
\end{equation*}
In particular, we have 
 \[
   y_b(t)=\int_{-\infty}^t \gamma_a^b(t-s)\varphi(x_0(s),y_0(s))\dif{s}.
 \]
Thus, from \eqref{eq:expanded},
\begin{multline*}
 \ddot x_0(t)=\dot y_0(t)=g\big(x_0(t),y_0(t),y_b(t)\big)+\lambda f\big(t,x_0(t)y_0(t)\big)=\\
            =g\left(x_0(t),\dot x_0(t), \int_{-\infty}^t\gamma_a^b(t-s) \varphi(x_0(s),\dot x_0(s))\dif{s}\right)+\lambda f\big(t,x_0(t),\dot x_0(t)\big),\quad \lambda\geq 0,
\end{multline*}
for all $t\in\R$, whence the assertion.
\end{proof}

\section{Branches of $T$-pairs}\label{sect:branches}

This section investigates the structure of the set of $T$-periodic solutions of \eqref{eq:ODEpert}. We begin by recalling some notation and basic facts.

Consider the following first-order parameterized ODE on $\R^k$:
\begin{equation}\label{eq:def-tpair}
 \dot x(t)=\g\big(x(t)\big)+\lambda\f\big(t,x(t)\big),
\end{equation}
where $\lambda \geq 0$, the maps
 $\g\colon \R^k\to\R^k$ and
$\f\colon \R\times\R^k\to\R^k$ are continuous
and $\f$ is $T$-periodic in the first variable.

We say that a pair $(\lambda,x)\in[0,\infty)\times C_T(\R^k)$ is a \emph{$T$-pair} (for \eqref{eq:def-tpair}) if $x$ is a $T$-periodic solution of \eqref{eq:def-tpair} corresponding to $\lambda$. If $\lambda=0$ and $x$ is constant then the $T$-pair is said \emph{trivial}. It is not hard to see that the trivial $T$-periodic pairs of \eqref{eq:def-tpair} correspond to the zeros of $\g$. 
{}From now on, given any $p \in \R^k$, we will denote by $\cl p$ the constant map $t \mapsto p$, $t \in \R$.
Given an open subset $\mathcal{O}$ of $[0,\infty)\X C_T(\R^k)$, we will denote by $\widetilde{\mathcal{O}} \subseteq \R^k$ the open set
$\widetilde{\mathcal{O}} := \big\{p:(0,\cl{p})\in\mathcal{O}\big\}$.

We have the following fact concerning the $T$-pairs of \eqref{eq:def-tpair}:

\begin{theorem}[\cite{FuSp97}]\label{thm:ramivecchio}
Let $\mathcal{O}$ be  open in $[0,\infty)\X C_T(\R^k)$, and assume that $\deg(\g,\widetilde{\mathcal{O}})$ is well defined and nonzero. Then there exists a connected set $\Gamma$ of nontrivial $T$-pairs in $[0,\infty)\X C_T(\R^k)$ whose closure in $\mathcal{O}$ is not compact and meets the set of trivial $T$-pairs contained in $\mathcal{O}$, namely the set: $\big\{(0,\cl{p})\in\mathcal{O}:\g(p)=0\big\}$.
\end{theorem}

Let us now go back to the second-order
equation \eqref{eq:ODEpert}.
We need some further notation.
The pairs $(\lambda,x) \in [0,+ \infty)\times C\sp 1_T(\R)$,
with $x\colon \R \to \R$ a $T$-periodic solution
of \eqref{eq:ODEpert} corresponding to $\lambda$, will be called \emph{$T$-forced pairs} (of \eqref{eq:ODEpert}), and we denote by $X$ the set of these pairs.
Among the $T$-forced pairs we shall consider as \emph{trivial} those of the type $(0,x)$ with $x$ constant.
Given an open subset $\Omega$ of $[0,\infty)\X C^1_T(\R)$, we will denote by $\widetilde{\Omega}   := \big\{u \in \R :(0,\cl{u})\in\Omega\big\}$, where by $\cl{u}$ we mean the constant map $t \mapsto u$, $t \in \R$.

As pointed out in Section \ref{sect:lct}, there is a correspondence between $T$-periodic solutions of the second order equation \eqref{eq:ODEpert} and of the first order system in \eqref{eq:expanded}. A similar correspondence holds between their ``$T$-periodic pairs'' in a sense that we are going to specify. We point out that, since equation \eqref{eq:expanded} can be seen as a special case of \eqref{eq:def-tpair}, in accordance with the notation introduced above,
a pair $(\lambda,\xi)\in[0,\infty)\times C_T(\R^{b+2})$ is called a \emph{$T$-pair}  if $\xi$ is a $T$-periodic solution of \eqref{eq:expanded} corresponding to $\lambda$. We also recall that the set of $T$-forced pairs of \eqref{eq:ODEpert} is regarded as a subset of $[0,+ \infty)\times C\sp 1_T(\R)$.

As a first remark we relate the corresponding ``trivial $T$-periodic pairs''.

\begin{remark} \label{rem:uno}
 Observe that if $(0,\cl{q})$, with $q \in\R$, is a trivial $T$-forced pair of \eqref{eq:ODEpert}, then $(0,\cl{\xi_0})$ is a $T$-pair for equation \eqref{eq:expanded}, where $\xi_0 \in\R^{b+2}$ is given by
\[
\xi_0:= (q,0,\varphi(q,0)\ldots,\varphi(q,0))
\]
that is,  by Remark \ref{rem:zero}, $\xi_0\in G^{-1}(0)$.
Conversely, any trivial $T$-pair of \eqref{eq:expanded}
must be of the form $(0,\cl{\xi_0})$,
 where $\xi_0 \in\R^{b+2}$ is such that $G(\xi_0)=0$.
Thus, again by Remark~\ref{rem:zero}, we have $\xi_0= (q,0,\varphi(q,0)\ldots,\varphi(q,0))$ for some $q \in\R$, and consequently $(0,\cl{q})$ is a trivial $T$-forced pair for \eqref{eq:expanded}.
\end{remark}

Let us now establish a general correspondence between the sets of $T$-forced pairs and of $T$-pairs, that preserves the notion of triviality.

Let $\mathcal{J}\colon [0,+ \infty)\times C\sp 1_T(\R) \to [0,\infty)\X C_T(\R^{b+2})$ be the map defined as follows:
\begin{equation}\label{def:omeo}
\mathcal{J}: (\lambda,x_0) \mapsto (\lambda,\xi)
\end{equation}
where $\xi:=(x_0,y_0,y_1\ldots,y_b)$ is given by
\begin{equation*}
\left\{\begin{array}{l}
 y_0(t):= \dot x_0(t),\\
  y_i(t):=\int_{-\infty}^t \gamma_a^i(t-s)\varphi(x_0(s),y_0(s))\dif{s},\quad i=1,\ldots,b.
 \end{array}\right.
\end{equation*}
As above, denote by $X$ the set of the
$T$-forced pairs of \eqref{eq:ODEpert}, and let $Y\subseteq [0,\infty)\X C_T(\R^{b+2})$ be the set of the
$T$-pairs of~\eqref{eq:expanded}.

\begin{lemma}\label{lemma:corrispRami}
 Let $\left.\mathcal{J}\right|_X\colon X\to Y$ be the restriction to $X$ of $\mathcal{J}$. Then $\left.\mathcal{J}\right|_X$ is a homeomorphism of $X$ onto $Y$,  which establishes a bijective correspondence between the trivial $T$-forced pairs in $X$ and the trivial $T$-pairs in $Y$.
\end{lemma}

\begin{proof}
Since $\varphi$ is locally Lipschitz continuous, the continuity of $\mathcal{J}$ is obtained by construction. Whence we get the continuity of the restriction $\left.\mathcal{J}\right|_X$. Injectivity and surjectivity of $\left.\mathcal{J}\right|_X\colon X\to Y$ follow from theorems \ref{thm:equivSol2} and \ref{thm:equivSol1}.

The inverse map $\left(\left.\mathcal{J}\right|_X\right)^{-1}$ can be seen merely as the projection onto the first two components, so it is continuous.

Finally the last part of the assertion follows by Remark \ref{rem:uno}.
\end{proof}

We are now in a position to state and prove our main result concerning the set of $T$-forced pairs of \eqref{eq:ODEpert}.

\begin{theorem}\label{thm:mainRamo}
Denote by $X\subseteq [0,\infty)\X C^1_T(\R)$ the set of $T$-forced pairs of \eqref{eq:ODEpert}. Let $\Phi$ be the real-valued function defined in \eqref{eq:defphi}, and
suppose that $(\alpha,\beta)\subseteq \R$ is such that $\Phi(\alpha)\cdot\Phi(\beta)<0$.
Let
$\Omega\subseteq [0,\infty)\X C^1_T(\R)$ be open and
such that $\widetilde \Omega =(\alpha,\beta)$.
Then, in $X\cap\Omega$ there is a connected subset $\Gamma$ of nontrivial $T$-forced pairs whose closure relative to $\Omega$ is not compact and intersects the set
\begin{equation} \label{eq:zeroinomega}
\big\{(0,\cl{u})\in\Omega:u\in (\alpha,\beta)\cap\Phi^{-1}(0)\big\}.
\end{equation}
\end{theorem}

\begin{proof}
Let $\mathcal{O}=\Omega\times C_T(\R\times\R^{b})$. Clearly, $\wt{\mathcal{O}}=\wt\Omega\times\R\times\R^{b}=(\alpha,\beta)\times\R\times\R^{b}$. Thus, Theorem \ref{thm:degree} implies that $\deg(G,\wt{\mathcal{O}})$ is well-defined and nonzero.
So, there exists a connected set $\Upsilon$ of nontrivial $T$-pairs for \eqref{eq:expanded} as in Theorem \ref{thm:ramivecchio}. Finally, by Lemma~\ref{lemma:corrispRami}, one sees that $\Gamma:=\mathcal{J}^{-1}(\Upsilon)$ is a connected set as in the assertion. 
\end{proof}

We remark that, as in the well-known  case of the resonant harmonic oscillator
\[
 \ddot x = -x +\lambda \sin t,
\]
this unbounded branch is possibly contained in the slice
$\{0\}\X C^1_T(\R)$.
We will discuss some conditions preventing this ``pathological'' situation in Section \ref{sec:ejecting}.

\section{Computation of the degree} \label{sect:degree}

\subsection{Brouwer degree in Euclidean spaces}
We will make use of the Brouwer degree in $\R^k$
in a slightly extended version (see e.g.\ \cite{BFPS03,DiMa21,Ni}) 
% \marginpar{\scriptsize\red{decidere se aggiungere altre citazioni}}
% that we briefly recall for the reader's convenience.

Let $U$ be an open subset of $\R^k$, $f$ a continuous $\R^k$-valued map whose domain contains the closure $\cl{U}$ of $U$, and $q \in \R^k$. We say that the triple $(f,U,q)$ is \emph{admissible (for the Brouwer degree)} if $f^{-1}(q)\cap U$ is compact.

The Brouwer degree is a function that to any admissible triple $(f,U,q)$ assigns an integer, denoted by $\deg(f,U,q)$ and called the \emph{Brouwer degree of $f$ in $U$ with target $q$}.
Roughly speaking, $\deg(f,U,q)$ is an algebraic count of the solutions in $U$ of the equation $f(p) = q$.
In fact, one of the properties of this integer-valued function is given by the following \emph{computation formula.} 

Recall that, if $f\colon U\to \mathcal \R^k$ is a $C\sp{1}$ map, an element $p\in U$ is said to be a \emph{regular point} (of $f$) if the differential of $f$ at $p$, $df_p$, is surjective.
Non-regular points are called \emph{critical (points)}.
The \emph{critical values} of $f$ are those points of $\R^k$ which lie in the image $f(C)$ of the set $C$ of critical points.
Any $q\in \mathcal \R^k$ which is not in $f(C)$ is a \emph{regular value}.
Therefore, in particular, any element of $\R^k$ which is not in the image of $f$ is a regular value.

\medskip\noindent
\textbf{Computation formula.}
\label{Computation Formula}
If $(f,U,q)$ is admissible, $f$ is smooth, and $q$ 
is a regular value for $f$ in $U$, then
\begin{equation}\label{eq:comp}
\deg(f,U,q)\; =\sum_{p \in f^{-1}(q) \cap U} \sign(df_p).
\end{equation}
This formula is actually the basic definition of the Brouwer degree, and the integer associated to any admissible triple $(g,U,r)$ is defined by
\begin{equation}\label{eq:approx}
\deg(g,U,r) := \deg(f,U,q),
\end{equation}
where $f$ and $q$ satisfy the assumptions of the Computation Formula and are, respectively, ``sufficiently close'' to $g$ and $r$.
It is known that this is a well-posed definition.

The more classical and well-known definition of Brouwer degree is usually given in the subclass of triples $(f,U,q)$ such that $f\colon\cl{U}\to\R^k$ is continuous, $U$ is bounded and $q \notin f(\partial U)$.
However, all the standard properties of the degree, such as homotopy invariance, excision, additivity, existence, are still valid in this more general context.

Since in this paper the target point $q$ will always be the origin, for the sake of simplicity, we will simply write $\deg(f,U)$ instead of $\deg(f,U,0)$.
In this context, we will say that an element $p \in f^{-1}(0)$ is a \emph{nondegenerate zero} (of $f$) if $\det (df_p) \neq 0$; this means, equivalently, that $p$ is a regular point. Observe that $\deg(f,U)$ can be regarded also as the degree (or characteristic, or rotation) of the map $f$ seen as a tangent vector field on $\R^k$.

\subsection{The degree of the map $G$}

The last part of this section is devoted to the proof of Theorem \ref{thm:degree}.
Roughly speaking we will relate the degree of the map $G$, defined in 
\eqref{eq:vectorFieldG} by
\[
G(u, v_0,v_1,\ldots,v_b)=\left(v_0,g(u,v_0,v_b),a(\varphi(u,v_0)-v_1),a(v_1-v_2)\ldots,a(v_{b-1}-v_b)\right),
\]
with that of the function $\Phi\colon \R\to\R$ in \eqref{eq:defphi}, given by 
\[
\Phi(u)=g(u,0,\varphi(u,0)).
\]

To simplify the computation of the degree, it is convenient to introduce the following map on $\R^{b+2}$:
\begin{equation*}\label{eq:defH}
 \mathcal{G}
 (u, v_0,v_1,\ldots,v_b):=\left(v_0,g(u,v_0,v_b),a(\varphi(u,v_0)-v_b),a(v_1-v_2)\ldots,a(v_{b-1}-v_b)\right) 
\end{equation*} 

\begin{lemma}\label{lem:homotopy}
Let $G$ and $\mathcal G$ be as above, and let $V\subseteq \R^{b+2}$ be open. Suppose that one of the maps $G$ or $\mathcal G$  is admissible for the Brouwer degree in $V$; then, so is the other and they are admissibly homotopic in $V$.
\end{lemma}

\begin{proof}
 For $(\lambda, u, v_0,v_1,\ldots,v_b)\in [0,1]\times V$, consider the map
\begin{multline*}
  \mathcal{H}(\lambda, u, v_0,v_1,\ldots,v_b)=\\=\left(v_0,g(u,v_0,v_b),a(\varphi(u,v_0)-[\lambda v_b+(1-\lambda)v_1]),a(v_1-v_2)\ldots,a(v_{b-1}-v_b)\right) ,
\end{multline*}
and observe that $\mathcal{H}(\lambda, u, v_0,v_1,\ldots,v_b)=0$ if and only if $(u, v_0,v_1,\ldots,v_b)\in V\cap G^{-1}(0) = V\cap \mathcal G^{-1}(0)$. Hence $\mathcal{H}$ is an admissible homotopy.
\end{proof}

Now we prove that if the zeros of $\Phi$ are
nondegenerate, so are those of $\mathcal G$.

\begin{lemma}\label{lem:nondegen}
All zeros of $\Phi$ are nondegenerate if and only if all zeros of  $\mathcal G$ are nondegenerate. 
\end{lemma}

\begin{proof}
First observe that, as in Remark \ref{rem:zero},
the zeros of $\mathcal G$ and those of $\Phi$ correspond.
In fact, if $\mathcal G(\bar u, v_0,v_1,\ldots,v_b)=0$ then $\Phi(\bar u)=0$, $v_0=0$ and $v_1=v_2=\ldots=v_b=\varphi(\bar u,0)=:\bar w$.
Conversely, for any $\bar u\in\Phi^{-1}(0)$, then $\mathcal G(\bar u,0,\bar w\ldots,\bar w)=0$.

Assume now that  $\bar \xi:=(\bar u,0,\bar w\ldots,\bar w)$ is a nondegenerate zero of $\mathcal G$. We compute the Jacobian matrix of $\mathcal G$ at $\bar \xi$ as follows:
  \begin{equation*}
   J\mathcal G_{\bar \xi}=\left(
          \begin{array}{cc|ccccc|c}
           0 & 1 & 0 & 0 & \hdots & \hdots & 0 & 0\\
           \D_1 g(\bar u,0,\bar w) & \D_2 g(\bar u,0,\bar w) & 0 & 0 & \hdots & \hdots& 0 & \D_3 g(\bar u,0,\bar w)\\
           a\,\D_1 \varphi(\bar u,0) & a\, \D_2 \varphi(\bar u,0) & 0 & 0 & \hdots &\hdots & 0 & -a\\ \hline
           0 & 0 & a & -a & 0& \hdots & 0 & 0 \\
           0 & 0 & 0 & a & -a & \ddots & 0 & 0 \\
           \vdots  & \vdots  & \ddots & \ddots &  \ddots & \ddots & \vdots& \vdots \\
           \vdots  & \vdots  & \ddots & \ddots &  \ddots & \ddots & \ddots& \vdots  \\
           0 &0 & 0 & 0 &\hdots & 0& a & -a  
          \end{array}
          \right)\;.
  \end{equation*}
By direct computation based on Laplace expansion we obtain
\[
\det J\mathcal G_{\bar \xi}=
(-a)^{b-1} \,\det
          \begin{pmatrix}
           0 & 1 & 0 \\
           \D_1 g(\bar u,0,\bar w) & \D_2 g(\bar u,0,\bar w) & 
           \D_3 g(\bar u,0,\bar w)\\
           a\D_1 \varphi(\bar u,0) & a\D_2 \varphi(\bar u,0) & -a  
          \end{pmatrix}=
          \]
          \[
= - (-a)^{b-1} \,\det
          \begin{pmatrix}
          \D_1 g(\bar u,0,\bar w)  & 
            \D_3 g(\bar u,0,\bar w)\\
           a\D_1 \varphi(\bar u,0)  & -a  
          \end{pmatrix}=
          \]
          \[
          = (-1)^{b-1}a^{b} \cdot\left[\D_1 g(\bar u,0,\bar w) +\D_3 g(\bar u,0,\bar w) \, \D_1 \varphi(\bar u,0)\right]= (-1)^{b-1}a^{b} \cdot \Phi'(\bar u).
\]
Thus $\bar u$ is a nondegenerate zero of the function $\Phi$.
The proof of the converse implication is analogous.
\end{proof}

Let us now compute the degree of the map $\mathcal G$.
\begin{lemma}\label{lem:comput}
 Assume $\Phi$ is admissible on an open set $U\subseteq \R$ and all its zeros are nondegenerate, then $\mathcal G$ is admissible in $U^*=U\times\R^{b+1}$, and 
 \[
  \deg(\mathcal G, U^*)=(-1)^{b-1}\deg(\Phi, U).
 \]
\end{lemma}

\begin{proof}
As in the proof of Lemma \ref{lem:nondegen} we have that all zeros of $\mathcal G$ are of the form $\bar \xi:=(\bar u,0,\bar w\ldots,\bar w)$ with $\Phi(\bar u)=0$. Moreover, they are all nondegenerate.
In particular, if $\Phi$ is admissible in $U$ then $\Phi^{-1}(0)\cap U$ is compact and so is $\mathcal G^{-1}(0)\cap U^*$, whence the admissibility of $\mathcal G$ in $U^*$.
Let now $\bar u \in \Phi^{-1}(0)\cap U$ be a nondegenerate zero of $\Phi$.
As in the proof of Lemma \ref{lem:nondegen} we have,
\[
\det J\mathcal G_{\bar \xi}=
 (-1)^{b-1}a^{b} \cdot \Phi'(\bar u)
\]
and, consequently,
\[
\sign\det\left(d\mathcal G_{\bar \xi}\right)=(-1)^{b-1}\sign\det \Phi'(\bar u)
\]
Thus, by formula \eqref{eq:comp},
\begin{multline*}
 \deg(\mathcal  G, U^*)=\sum_{\bar \xi\in\mathcal G^{-1}(0)\cap U^*}\sign\det\left(d\mathcal G_{\bar \xi}\right)=\\
 =\sum_{\bar u\in\Phi^{-1}(0)\cap U}(-1)^{b-1}\sign\det \Phi'(\bar u)
=(-1)^{b-1}\deg(\Phi, U),
\end{multline*}
whence the assertion.
\end{proof}
The above Lemmas imply now the assertion of Theorem \ref{thm:degree}.

\begin{proof}[Proof of Theorem \ref{thm:degree}]
Let  $(\alpha,\beta)\subseteq \R$ be as in the assertion of Theorem \ref{thm:degree}.
As in equality \eqref{eq:approx}, we can assume that all zeros of $\Phi$ in $(\alpha,\beta)$ are nondegenerate. The assertion now follows from Lemma \ref{lem:comput} and Lemma \ref{lem:homotopy}.
\end{proof}

\section{Ejecting points and small perturbations}\label{sec:ejecting}

As we have seen by the simple example at the end of Section \ref{sect:branches}, it may happen that the branch $\Gamma$ of $T$-forced pairs of equation \eqref{eq:ODEpert}, as in the assertion of Theorem~\ref{thm:mainRamo}, is completely contained in the slice $\{0\}\times C_T^1(\R)$. In this section we show simple conditions ensuring that this is not the case. Such conditions are based on the key notion of \emph{ejecting set} (see, e.g.,~\cite{FPS00}).
 
 Let us first introduce the following notation: let $\Upsilon$ be a subset of $[0,\infty )\times C_T^1(\R)$. Given $\lambda \geq 0$, let $\Upsilon_{\lambda}$ be the slice $\big\{x\in C_T^1(\R):(\lambda ,x)\in \Upsilon \big\}$.
Below, we adapt to our context the definition of ejecting set of \cite{FPS00}:
 \begin{definition}\label{de:ejecting}
Let $X\subseteq [0,\infty )\times C_T^1(\R)$ be the set of $T$-forced pairs of~\eqref{eq:ODEpert}, and let $X_0$ be the slice of $X$ at $\lambda=0$. We say that $A\subset X_0$ is an \emph{ejecting set} (for $X$) if it is relatively open in $X_0$ and there exists a connected subset of $X$ which meets $A$ and is not contained in $X_0$. In particular, when $A=\{p_0\}$ is a singleton we say that $p_0$ is an \emph{ejecting point}.
 \end{definition}
 
 We now discuss a sufficient condition for an isolated  point of $\Phi$ to be ejecting for the set $X$ of $T$-forced pairs of \eqref{eq:ODEpert}. This condition  is based on a result by J.\ Yorke (\cite{Yo69}, see also \cite{BuMa87}) concerning the period of solutions of an autonomous ODE with Lipschitz continuous right-hand side. 
 
 We point out that an analysis of  \eqref{eq:expanded}, for $\lambda=0$, linearized at its zeros leads to a different, not entirely comparable, approach. A discussion of the latter technique, which is based on the notion of $T$-resonance and requires the knowledge of the spectrum of the linearized equation is outside the scope of the present paper (see, e.g., \cite[Ch.\ 7]{Cr94} and \cite[Ch.\ 2]{Cr64}, a similar idea can be traced back to Poincar\'e see, e.g., \cite{Maw14}; see also \cite{BiSp15} for an application of this idea to multiplicity results). 

The aforementioned result of Yorke is the following:
 
 \begin{theorem}[\cite{Yo69}]\label{th:yorke}
Let $\xi$ be a nonconstant $\tau$-periodic solution of 
\[
 \dot x = \mathcal F(x)
\]
where $\mathcal F\colon W\subseteq\R^k\to\R^k$ is Lipschitz of constant $L$. Then the period $\tau$ of $\xi$ satisfies $\tau\geq 2\pi/L$. 
 \end{theorem}
 
 Inequalities on the period like that of Theorem \ref{th:yorke} have been object of study in different contexts, see e.g., the introduction of \cite{BuMa92} for an interesting discussion.
 
 Consider the map $G$ defined in \eqref{eq:vectorFieldG}, and suppose that $g$ and $\varphi$ are Lipschitz continuous. Then, so is $G$. Let $L_G$ be the Lipschitz constant of $G$.
 
{}From Theorem \ref{th:yorke} we immediately deduce the following facts concerning the branch of $T$-forced pairs 
of \eqref{eq:ODEpert},
given by Theorem \ref{thm:mainRamo}:

 \begin{corollary}\label{co:emanation1}
 Let $I$, $\Omega$ and $\Phi$ be as in Theorem \ref{thm:mainRamo}. If the period $T$ of the forcing term $f$ satisfies $T<2\pi/L_G$, then the set $\Gamma$ of Theorem \ref{thm:mainRamo} cannot intersect the slice $\{0\}\times C_T^1(\R)$.
 \end{corollary}
 
\begin{proof}
 Assume by contradiction that there exists a nontrivial $T$-forced pair
 $(0,x_0)\in\Gamma$. Then $x_0$ corresponds (in the sense of Theorem \ref{thm:equivSol2}) to a $T$-periodic solution of $\dot \xi=G(\xi)$ with $T<2\pi/L_G$,  violating Theorem \ref{th:yorke}. 
\end{proof}
A similar argument yields the following result:
\begin{corollary}\label{co:emanation2}
 Let $I$, $\Omega$ and $\Phi$ be as in Theorem \ref{thm:mainRamo}. 
Let $p\in I$ be an isolated zero of $\Phi$ and assume that $G$ is locally Lipschitz in a neighborhood $W$ of the corresponding zero $P=\big(p,0,\varphi(p,0),\ldots,\varphi(p,0)\big)$ of $G$. Suppose that a set $\Upsilon$ of $T$-pairs for \eqref{eq:expanded} contains $\left(0,\cl{P}\right)$. %this point.
Let $\Lambda=\mathcal{J}^{-1}(\Upsilon)$ be the corresponding set of $T$-forced pairs of \eqref{eq:ODEpert}. As above, call $L_G$ the Lipschitz constant of $G$ in $W$. Then, the trivial $T$-forced pair $(0,\bar p)$ is isolated in the slice $\Lambda_0$. In particular, if $T<2\pi/L_G$, the set $\Lambda$ cannot be contained in the slice $\{0\}\times C_T^1(\R)$.
\end{corollary}

Observe that, unlike $\Gamma$ in Corollary \ref{co:emanation1}, the set $\Upsilon$ in Corollary \ref{co:emanation2} does not necessarily consist of nontrivial $T$-pairs. Consequently, $\Lambda$ does not necessarily consist of nontrivial $T$-forced pairs as well. Notice also that since $p$ is an isolated zero of $\Phi$ then $\left(0,\cl{P}\right)$, %$P=\big(p,0,\varphi(p,0),\ldots,\varphi(p,0)\big)$, 
is isolated in the set of trivial $T$-pairs of \eqref{eq:expanded} and, similarly, $(0,\bar p)$ is isolated in the set of trivial $T$-forced pairs of \eqref{eq:ODEpert}.

\begin{proof}[Proof of Corollary \ref{co:emanation2}]
By construction of the map $\mathcal{J}$, the $T$-forced pairs of \eqref{eq:ODEpert} that lie in the slice $\{0\}\times C_T^1(\R)$ corresponds bijectively to the $T$-pairs of \eqref{eq:expanded} contained in $\{0\}\times C_T(\R^{b+2})$.

Theorem \ref{th:yorke} implies that there are not nonconstant $T$-periodic solutions of \eqref{eq:expanded} contained in $W$. Thus $\Upsilon$ cannot be contained in $\{0\}\times C_T(\R^{b+2})$.  Hence, the same is true for $\Gamma=\mathcal{J}^{-1}(\Upsilon)$.
\end{proof}

\begin{remark}\label{re:emC1}
When $G$ is $C^1$ in a neghborhood of a nondegenerate zero, then it is locally Lipschitz in this neighborhood.
Then by Corollary \ref{co:emanation2} we have that
\emph{if the frequency of the forcing term is sufficiently high, then a nondegenerate zero of the unperturbed vector field is necessarily an ejecting point of nontrivial $T$-forced pairs of \eqref{eq:ODEpert}.}
\end{remark}

\section{Ejecting sets and multiplicity}

This section is devoted to the illustration of sufficient conditions on $f$, $g$, and $\varphi$ yielding the multiplicity of  $T$-periodic solutions of \eqref{eq:ODEpert} for $\lambda> 0$ small.

 The following lemma of \cite{FPS00} about the notion of ejecting set (or point) is crucial to get our multiplicity result. As in the previous section, we reformulate the lemma in our setting. As above,  let $G$ be as in \eqref{eq:vectorFieldG}, and let $X\subseteq [0,\infty )\times C_T^1(\R)$ be the set of nontrivial $T$-forced pairs of \eqref{eq:ODEpert}.

\begin{lemma}[{\protect\cite[Lemma 3.1]{FPS00}}]
\label{le:conn1}
Let $K$ be a relatively open compact subset of the slice $X_0$ of $X$ at $\lambda=0$. Then, for any sufficiently small open neighborhood $U$ of $K$ in $\{0\}\times C_T^1(\R)$, there exists a positive number $\delta$ such that
\[
X\cap\big( [0,\delta ]\times\partial U\big) =\emptyset\;.
\]
\end{lemma}

 Corollary \ref{co:emanation2}, Theorem \ref{thm:mainRamo} and Lemma \ref{le:conn1} yield the following result about multiple periodic solutions of \eqref{eq:ODEpert}.
 
 \begin{theorem}\label{thm:mult0}
  Let $f$, $g$, and $\varphi$ be as in \eqref{eq:ODEpert}. Let $I\subseteq\R$ be an open interval and assume that $\Phi$ changes sign at the isolated zeros $p_1,\ldots, p_n\in I$. For $i=1,\ldots,n$, put $P_i=\big(p_i,0,\varphi(p_i,0)\ldots,\varphi(p_i,0)\big)$.  Assume that $g$ and $\varphi$ are such that $G$ is Lipschitz with constant $L$ on a neighborhood of $\{P_1,\ldots,P_n\}$ and the period $T$ of $f$ satisfies $T<2\pi/L$. Then:
  \begin{enumerate}
   \item\label{pr:mult0-1} For $i=1,\ldots,n$, there exist connected sets $\Gamma_i$ of nontrivial $T$-pairs of \eqref{eq:ODEpert} emanating from $(0,\bar p_i)$;
   \item\label{pr:mult0-2} There exists $\lambda _{*}>0$ such that the projection of each $\Gamma_i$ on the first component contains $[0,\lambda_*]$. Consequently, \eqref{eq:ODEpert} has at least $n$ solutions $x_1^{\lambda},\ldots,x_n^{\lambda}$ of period $T$ for $\lambda\in[0,\lambda_*)$;
   \item\label{pr:mult0-3} The images of $x_1^{\lambda},\ldots,x_n^{\lambda}$ are pairwise disjoint.
  \end{enumerate}
\end{theorem}

\begin{proof}
 Assertion \eqref{pr:mult0-1} follows immediately from Theorem \ref{thm:mainRamo} applied to pairwise disjoint open subintervals $I_i$, $i=1,\ldots,n$, isolating the zero $p_i$.
 
 To prove assertion \eqref{pr:mult0-2} let, for  $i=1,\ldots,n$, $U_i= C_T^1(I_i)$ and $K_i=(0,\bar p_i)$. By Corollary \ref{co:emanation2}, $K_i$ is isolated, hence relatively open and compact (being a singleton). Then, Lemma \ref{le:conn1} yields positive $\delta_i$'s such that$X\cap\big([0,\delta]\times\partial U_i\big)=\emptyset$.
 
 Recall now that $X$ is locally complete. Since the closure of $\Gamma_i$ in $X$ meets $(0,\bar p_i)$ and is not contained in any compact subset of $X$, it is not difficult using Ascoli-Arzel\`a theorem, to prove that that the projection of $\Gamma_i$ on its first component is $[0,\delta_i]$. The proof is completed by taking $\lambda_*=\min\{\delta_1,\ldots,\delta_n\}$. 
 
 To prove the last assertion, observe that by \eqref{pr:mult0-2} there exists $\lambda_*>0$ such that each $\{\lambda\}\times U_i$, contains at least one $T$-forced pair, say $(\lambda,x_i^\lambda)$ for any  $\lambda\in[0,\lambda_*)$. Hence, for $j,k=1,\ldots,n$ and $j\neq k$, the images of $x_j^\lambda$ and $x_k^\lambda$ are confined to the disjoint sets $I_i$ and $I_k$.
 \end{proof}
 
 Restricting to a neighborhood of the set of zeros of $G$ corresponding to $p_1,\ldots, p_{n}$, we can give a somewhat less technical and perhaps more elegant formulation of our multiplicity result.
 
 \begin{corollary}
  Let $f$, $g$, and $\varphi$ be as in \eqref{eq:ODEpert}, with $g$ and $\varphi$ locally Lipschitz. Assume that $\Phi$ changes sign at the isolated zeros $p_1,\ldots, p_{n}\in I$. Then, for sufficiently high frequency of the perturbing term $f$ and sufficiently small $\lambda>0$,  equation \eqref{eq:ODEpert} has at least $n$ solutions of period $T$ whose images are pairwise disjoint.
 \end{corollary}

\section{Visual representation of branches}
 We briefly discuss a method allowing to represent graphically the infinite dimensional set of $T$-forced pairs of \eqref{eq:ODEpert}. In other words, as in \cite{BiSp15}, we create a homeomorphic finite dimensional image of the set $\Gamma$ yielded by Theorem \ref{thm:mainRamo} and show a graph of some relevant functions of the point of $\Gamma$ as, for instance, the sup-norm or the diameter of the orbit of the solution $x$ in any $T$-forced pair $(\lambda,x)$.
 
 In this section we assume $g$ and $\varphi$ as well as the perturbing term $f$ to be at least Lipschitz continuous, so that continuous dependence on initial data of \eqref{eq:expanded} holds.
 
 Let us consider the set
\[
 S=\left\{ (\lambda,q,p_0,\ldots,p_b)\in[0,\infty)\times\R^{b+2}\;\;\left|\;\parbox{0.42\linewidth}{\small $(q,p_0,\ldots,p_b)$ is an initial condition at $t=0$ for a $T$-periodic solution of \eqref{eq:expanded}}\right.\right\}
\]
The elements of $S$ are called \emph{starting points} for \eqref{eq:expanded}. A starting point $(\lambda,q,p_0,\ldots,p_b)$ is \emph{trivial} when $\lambda=0$ and the solution of \eqref{eq:expanded} starting a time $t=0$ from  $(q,p_0,\ldots,p_b)$ is constant.
 
By uniqueness and continuous dependence on initial data the map $\mathfrak{p}\colon Y\to S$ given by
\[
(\lambda,x_0,y_0,\ldots.y_b)\mapsto\big(\lambda,x_0(0),y_0(0),\ldots.y_b(0)\big)
\]
is a homeomorphism that establishes a correspondence between trivial $T$-pairs and trivial starting points. Thus, the composition $\mathfrak{h}=\mathfrak{p}\circ\mathcal{J}^{-1}\colon X\to S$ is as well a homeomorphism that establishes a correspondence between trivial $T$-forced pairs and trivial starting points.  In other words,  $\Sigma:=\mathfrak{h}(\Gamma)\subseteq [0,\infty)\times\R^{b+2}$ is the desired homemorphic image of $\Gamma$.
 
\begin{figure}[ht!]
 \begin{tabular}{@{}cc@{}}
  \subfigure[sup-norm of $x$ for all $(\lambda,x)\in X$]{\includegraphics[width=0.48\linewidth]{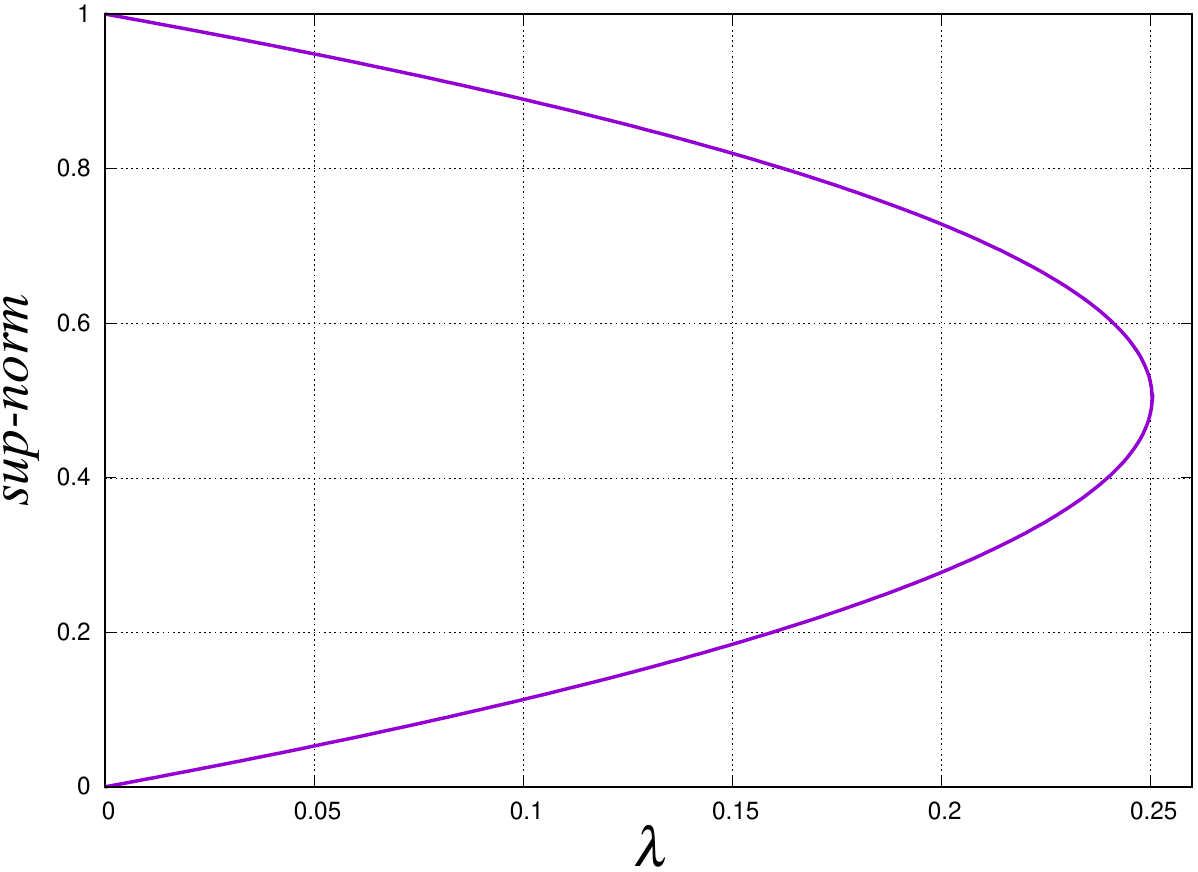}} &
  \subfigure[diameter of the orbits of $x$ for all $(\lambda,x)\in X$ (i.e., $\max_{t\in [0,T]} x(t)- \min_{t\in [0,T]} x(t)$)]{\includegraphics[width=0.48\linewidth]{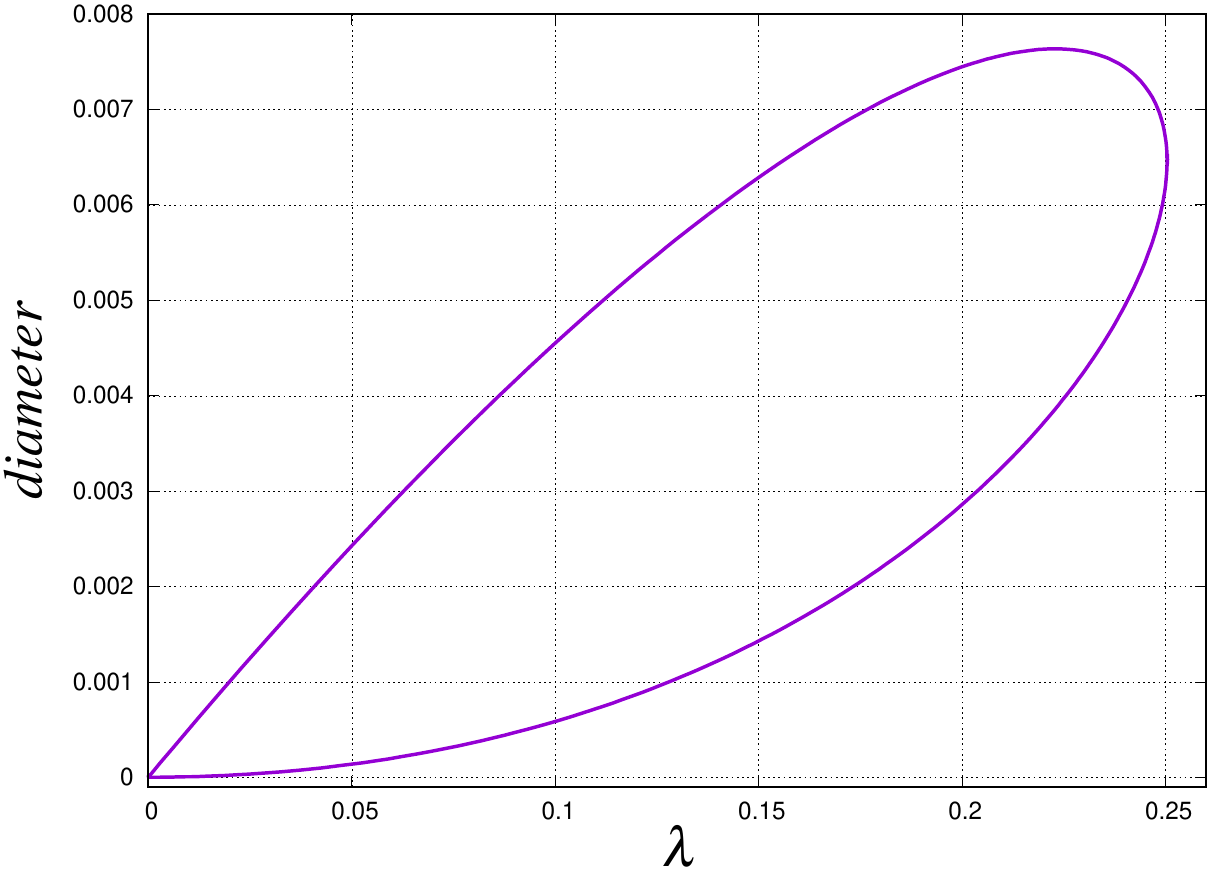}}
 \end{tabular}
\caption{Sup-norm and diameter of points in $\Gamma$}\label{fig:1}
\end{figure}
\begin{example}
Consider the following equation:
\begin{equation}\label{eq:example}
 \ddot x(t) = -x(t)\left(1+\int_{-\infty}^t\gamma_2^2(t-s)\big(\dot x(s)-x(s)\big)\dif{s}\right)+\lambda\big(1+x(t)\sin(2\pi t)\big),
\end{equation}
with $\lambda\geq0$. Here, $T=1$, $g(\xi,\eta,\zeta)=-\xi(1+\zeta)$, $\varphi(p,q)=q-p$, so that
\[
\Phi(u)=g\big(u,0,\varphi(u,0)\big)=u(1-u),
\]
and
\[
 G(u,v_0,v_1,v_2)=\big(v_0,-u(1+v_2),2(v_0-u-v_1),2(v_1-v_2)\big).
\]
Clearly, $\Phi$ changes sign at the zeros $0$ and $1$, the corresponding zeros of $G$ being $P_0:=(0,0,0,0)$ and $P_1:=(1,0,-1,-1)$. Observe that $G$, in suitably small neighborhoods of $P_0$ and $P_1$, is Lipschitz with constant smaller than $2$. Thus, by Corollary \ref{co:emanation2}, $(0,\cl 0)$ and $(0,\cl 1)$ are ejecting for the set $X$ of $T$-forced pairs of \eqref{eq:example}. Let
\begin{equation*}
\Omega_0=[0,\infty)\times C_T^1(\R)\setminus\{(0,\cl 1)\},\qquad
\Omega_1=[0,\infty)\times C_T^1(\R)\setminus\{(0,\cl 0)\}.
\end{equation*}
Theorem \ref{thm:mainRamo} implies the existence of connected sets $\Gamma_0$ and $\Gamma_1$ of nontrivial $T$-forced pairs of \eqref{eq:example} whose closures, respectively, emanate from $(0,\cl 0)$ and $(0,\cl 1)$ and are not contained in any compact subset of $\Omega_0$ and $\Omega_1$. Figure \ref{fig:1} shows sup-norm and diameter of the solutions of the $T$-forced pairs of \eqref{eq:example}. Figure \ref{fig:2}, instead shows the projections of $\Sigma$ on the plane $(\lambda,q)$ and on the 3-dimensional space $(q,p_0,\lambda)$. Indeed, Figures \ref{fig:1} and \ref{fig:2} suggest that  $\Gamma_1=\Gamma_0$. Observe also that, by Theorem \ref{thm:mult0} there are two periodic solutions of \eqref{eq:example} for small $\lambda>0$. The figures suggest that the value of $\lambda_*$ in Theorem \ref{thm:mult0} is about $0.25$.
\end{example}

\begin{figure}[h!]
 \begin{tabular}{@{}cc@{}}
  \subfigure[projection of $X$ on the plane $(\lambda,q)$]{\includegraphics[width=0.48\linewidth]{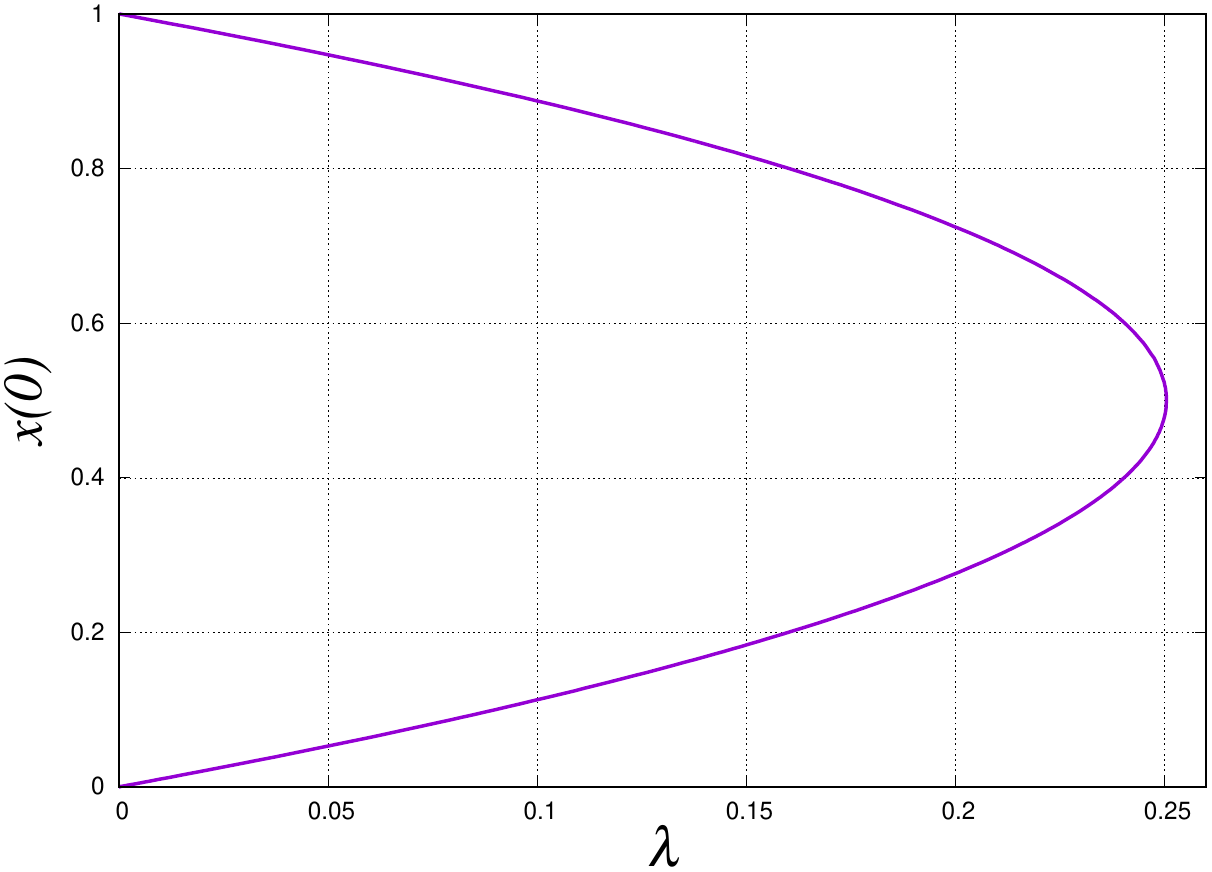}} &
  \subfigure[projection of $X$ on the space $(q,p_0,\lambda)$]{\includegraphics[width=0.48\linewidth]{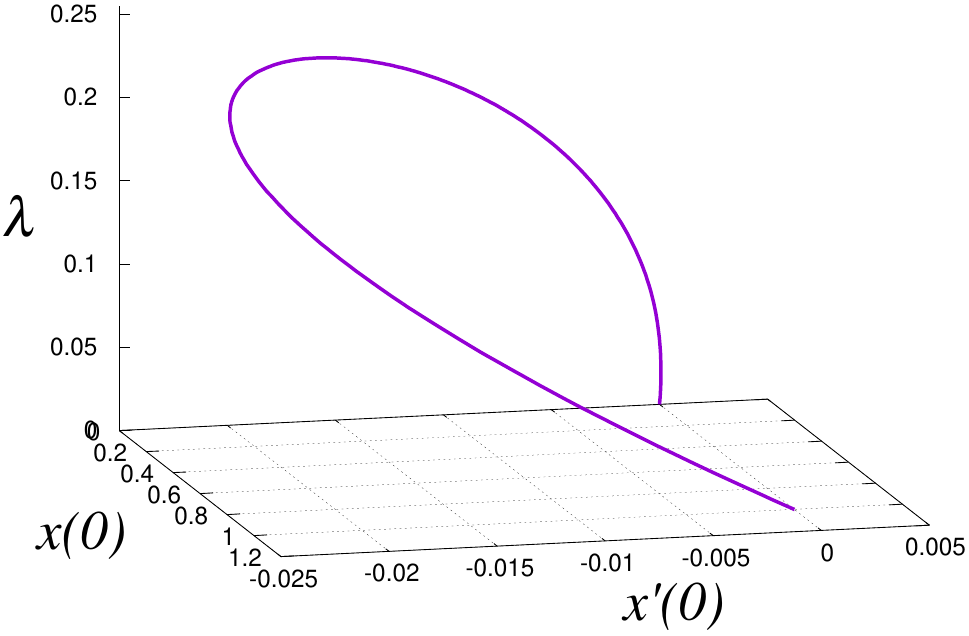}}
 \end{tabular}
\caption{Inital values and speed of $(\lambda,x)\in \Gamma$}\label{fig:2}
\end{figure}

\end{document}